\documentclass[a4paper]{article}

\usepackage[margin=25mm]{geometry}
\usepackage{amsmath}
\usepackage{amsfonts}
\usepackage{amssymb}
\usepackage{amsthm}
\usepackage{mathbbol}
\usepackage{graphicx}
\usepackage{color}
\usepackage{enumerate}

\newtheorem{assum}{Assumption}
\newtheorem{defi}{Definition}
\newtheorem{prop}{Proposition}
\newtheorem{coro}{Corollary}

\newtheorem{lem}{Lemma}
\newtheorem{rem}{Remark}

\def\ds{\displaystyle}

\title{An optimal feedback control that minimizes the epidemic peak in the SIR model under a budget constraint}

\author{{\sc Emilio Molina$^{1,2}$ and Alain Rapaport$^2$}\\[2mm]
	$^1$ DIM \& CMM, Santiago-de-Chile, Universidad de Chile\\
	and LJLL, Sorbonne Universit\'e, INRIA, France\\[2mm]
	$^2$ MISTEA, Universit\'e Montpellier, INRAE, Institut Agro, France\\[2mm]
	{\em E-mail addresses:} {\tt emolina@dim.uchile.cl}, {\tt alain.rapaport@inrae.fr}
}

\date{\today}

\begin{document}
	
	\maketitle
	
	\begin{abstract}
		We give the explicit solution of the optimal control problem which consists in minimizing the epidemic peak in the SIR model when the control is an attenuation factor of the infectious rate, subject to a $L^1$ budget constraint. The optimal strategy is given as a feedback control which consists in an singular arc maintaining the infected population at a constant level until the immunity threshold is reached, and no intervention outside the singular arc.\\
		
	\end{abstract}
	
	\section{Introduction}
	
	Since the pioneer work of Kermack and McKendrick \cite{KermackMcKendrick}, the SIR model has been very popular in epidemiology, as the basic model for infectious diseases with direct transmission (see for instance \cite{SIR,Li} as introductions on the subject). It retakes great importance nowadays due to the recent coronavirus pandemic. In face of a new pathogen, non-pharmaceutical interventions (such as reducing physical distance in the population) are often the first available means to reduce the propagation of the disease, but this has economical and social prices... In \cite{Morris,Lobry}, the authors underline the need of control strategies for epidemic mitigation by "flattering the epidemic curve", rather than eradication of the disease that might be too costly. Several works have applied the optimal control theory considering interventions as a control variable that reduces the effective transmission rate of the SIR model, and studied optimal strategies with criteria based on running and terminal cost over fixed finite interval or infinite horizon \cite{Behncke,Bolzoni2017,Bolzoni2019,KantnerKoprucki,PalmerZabinskuLiu,Bliman20,Caulkins,Freddi,Ketcheson,Bliman21}. However, the highest peak of the epidemic appears to be the highly relevant criterion to be minimized (especially when there is an hospital pressure to save individuals with severe forms of the infection). In \cite{Morris}, the authors studied the minimization of the peak of the infected population under the constraint that interventions occur on a single time interval of given duration. In the present work, we consider the same criterion, but under a budget constraint on the control (as an integral cost) that we believe to be more relevant as it takes into account the strength of the interventions and does not impose an {\em a priori} single time interval of given length for the interventions to take place (although we have been able to prove that the optimal solution consists indeed in having interventions on a single time interval but with a control strategy different that the one obtained in \cite{Morris}). Let us also mention a more recent work \cite{AvramFreddiGoreac} that considers a kind of "dual" problem, which consists in minimizing an integral cost of the control under the constraint that the epidemic stays below a prescribed value and an additional constraint on the state at a fixed time. The structure of the optimal strategy given by the authors in \cite{AvramFreddiGoreac} is similar to the one we obtained without having to fix a time horizon and a terminal constraint. All the cited works rely on numerical methods to provide the effective control. Here, we give an explicit analytical expression of the optimal control.

	Let us stress that optimal control problems with maximum cost are not in the usual  Mayer, Lagrange or Bolza forms of the optimal control theory \cite{Cesari}, for which the necessary optimality conditions of Pontryagin's Principle apply, but fall into the class of optimal control with $L^\infty$ criterion, for which characterizations  have been proposed in the literature mainly in terms of the value function (see for instance \cite{BarronIshii}).
	Although necessary optimality conditions and numerical procedures have been derived from theses characterizations (see for instance \cite{Barron,DiMarcoGonzalez1999}), these approaches remain quite difficult and numerically heavy to be applied on concrete problems. On another hand, for minimal time problems with planar dynamics linear with respect to the control variable, comparison tools based on the application of the Green's Theorem have shown that it is possible to dispense with the use of necessary conditions to prove the optimality of a candidate solution \cite{HermesLaSalle}. Although our criterion is of different nature, we show in the present work that it is also possible to implement this approach for our problem.
	
	The paper is organized as follows. In the next section, we posit the problem of peak minimization to be  studied. In Section \ref{sectionNSN}, we define a class of feedback strategies that we called "NSN", and give some preliminary properties. Section \ref{sectionProof} proves that the existence of an NSN strategy which is optimal for our problem, and makes it explicit. Finally, Section \ref{sectionSimu} illustrates the optimal solutions on numerical simulations and discusses about the optimal strategy.
	
	\section{Definitions and problem statement}
	\label{sectionDef}
	
	We consider the SIR model
	\begin{equation}
	\label{model}
	\left\{\begin{array}{l}
	\dot S = - \beta SI(1-u)\\
	\dot I = \beta SI(1-u)-\gamma I\\
	\dot R = \gamma I
	\end{array}\right.
	\end{equation}
	where $S$, $I$ and $R$ denotes respectively the proportion of susceptible, infected and recovered individuals in a population of constant size. The parameters $\beta$ and $\gamma$ are the transmission and recovery rates of the disease. The control $u$, which belongs to $U:=[0,1]$, represents the efforts of interventions by reducing the effective transmission rate. For simplicity, we shall drop in the following the $R$ dynamics.
   Throughout the paper, we shall assume that the {\em basic reproduction number} ${\cal R}_0$ is larger than one, so that an epidemic outbreak may occur.
	\begin{assum}
		\label{assumR0}
		\[
		{\cal R}_0:= \frac{\beta}{\gamma} > 1
		\]
	\end{assum}
	For a positive initial condition $(S(0),I(0))=(S_0,I_0)$ with $S_0+I_0\leq 1$, we consider the optimal control problem which consists in minimizing the epidemic peak under a budget constraint
	\begin{equation}
		\label{pbpeak}
	\inf_{u(\cdot) \in {\cal U}} \max_{t \geq 0} I(t)
	\end{equation}
	where ${\cal U}$ denotes the set of measurable functions $u(\cdot)$ that take values in $U$ and satisfy the $L^1$ constraint
	\[
	\int_0^{+\infty} u(t) dt \leq Q
	\]
	\begin{rem}
	From equations \eqref{model}, one can easily check that the solution $I(t)$ tends to zero when $t$ to $+\infty$ whatever is the control $u(\cdot)$, so that the supremum of $I(\cdot)$ over $[0,+\infty)$ in \eqref{pbpeak} is reached.
	\end{rem}

Equivalently, one can consider the extended dynamics.
	\begin{equation}
		\label{sys}
		\left\{\begin{array}{l}
			\dot S = - \beta SI(1-u)\\
			\dot I = \beta SI(1-u)-\gamma I\\
			\dot C= -u
		\end{array}\right.
	\end{equation}
with the initial condition $(S(0),I(0),C(0))=(S_0,I_0,Q)$ and the state constraint
\begin{equation}
	\label{target}
C(t)\geq 0, \quad t \geq 0
\end{equation}
A solution of \eqref{sys} is {\em admissible} if the control $u(\cdot)$ takes its values in $U$ and the condition \eqref{target} is fulfilled.

\section{The NSN feedback}
\label{sectionNSN}

	Let us denote the {\em immunity threshold}
	\[
	S_{h}:={\cal R}_0^{-1}=\frac{\gamma}{\beta} < 1
	\]
	Note that $S(\cdot)$ is a non increasing function and that one has  $\dot I\leq 0$ when $S \leq S_{h}$, whatever is the control. If $S_0\leq S_{h}$, the maximum of $I(\cdot)$ is thus equal to $I_0$ for any control $u(\cdot)$, which solves the optimal control problem. We shall now consider that the non-trivial case.
	\begin{assum}
		\label{assumS0}
		\[
		S_0 > S_{h}
		\]
	\end{assum}
Under this assumption, we thus know that for any admissible solution, the maximum of $I(\cdot)$ is reached for $S\geq S_{h}$.  For the control $u=0$, one can easily check that following property is fulfilled
	\begin{equation}
	    \label{invariant}
	S(t)+I(t)-S_h\log(S(t))=S_0+I_0-S_h\log(S_0), \quad t >0
	\end{equation}
	and the maximum of $I(\cdot)$ is then reached for the value
	\[
	I_{h}:= I_0 + S_0 -S_{h}-S_{h}\log\left(\frac{S_0}{S_{h}}\right)
\]
	We define the "NSN" (for null-singular-null) strategy as follows.
	\begin{defi}
	For $\bar I \in [I_0,I_{h}]$, consider the feedback control
	\begin{equation}
	\label{feedback}
	\psi_{\bar I}(I,S):=\begin{cases}
	1- \frac{S_{h}}{S} & \mbox{if } I=\bar I \mbox{ and  } S> S_{h}\\0 & \mbox{otherwise}
	\end{cases}
	\end{equation}
	We denote the $L^1$ norm associated to the NSN control
\[
{\cal L}(\bar I):= \int_0^{+\infty} u^{\psi_{\bar I}}(t)dt, \quad \bar I \in [I_0,I_h]
\]
where $u^{\psi_{\bar I}}(\cdot)$ is the control generated by the feedback \eqref{feedback}.
\end{defi}
This control strategy consists in three phases:
\begin{enumerate}
	\item no intervention until the prevalence $I$ reaches $\bar I$ (null control),
	\item maintain the prevalence $I$ equal to $\bar I$ until $S$ reaches $S_h$ (singular control),
	\item no longer intervention when $S>S_h$ (null control)
\end{enumerate}
\begin{rem}
	\label{remswitch}
	There is no switch of the control between phases 2 and 3, because $u(t)$ tends to $zero$ when $S(t)$ tends to  $S_h$, according to expression \eqref{feedback}.
\end{rem}

\bigskip

One can check straightforwardly the the following properties are fulfilled.

 	\begin{lem}
 		\label{lemIbar}
 	For any $\bar I \in [I_0,_h]$, the maximal value of the control $u^{\psi_{\bar I}}(\cdot)$ is given by
 	\[
 	u_{max}(\bar I):=1-\frac{S_h}{\bar S} < 1
 	\]
 	where $\bar S$ is solution of
 	\[
 	\bar S - S_h\log \bar S = S_0 + I_0 -S_h \log S_0 - \bar I
 	\]
 	Moreover, any solution given by the NSN strategy verifies
   \[
 	\max_{ t \geq 0} I(t) = \bar I
 	\]
 \end{lem}

\section{Optimal strategy}
\label{sectionProof}

We first show that the function ${\cal L}$ can be made explicit.
	
	\begin{prop} 
		\label{propcalC}
		One has
		\begin{equation}
			\label{calL}
		{\cal L}(\bar I)=\frac{I_h-\bar I}{\beta S_h\bar I}, \quad \bar I \in [I_0,I_h]
		\end{equation}
\end{prop}

		\begin{proof}
	Note first that whatever is $\bar I$, $S(\cdot)$ is decreasing with the control \eqref{feedback}. One can then equivalently parameterize the solution $I(\cdot)$, $C(\cdot)$ by 
	\[
	\sigma(t):=S_0-S(t)
	\]
	instead of $t$. Posit $\sigma_h:=\sigma(t_h)=S_0-S_h$.
	
	\medskip

	As long as $I<\bar I$, one has $u=0$ which gives
\[
\left\{\begin{array}{l}
\ds \frac{dI}{d\sigma}=f(\sigma):=1-\frac{S_{h}}{S_0-\sigma}>0\\[3mm]
\ds	\frac{dC}{d\sigma}=0
\end{array}\right.
\]
Remind, from the definition of $I_h$, that the solution $I(\cdot)$ with $u=0$ reaches $I_h$ in finite time. Therefore, one can define the number
\[
\bar \sigma := \inf \{ \sigma \geq 0, \; I(\sigma)=\bar I \} \leq \sigma_h
\]
which verifies
\begin{equation}
	\label{barsigma}
	\int_0^{\bar\sigma} f(\sigma)\,d\sigma=\bar I-I_0
\end{equation}
For $\sigma \in [\bar \sigma,\sigma_h]$, one has $u=1-S_h/S$, that is
\[
\left\{\begin{array}{l}
\ds \frac{dI}{d\sigma}=0\\[3mm]
\ds	\frac{dC}{d\sigma}=-\frac{1}{\beta S_h \bar I}\left(1-\frac{S_{h}}{S_0-\sigma}\right)=-\frac{f(\sigma)}{\beta S_h \bar I}<0
\end{array}\right.
\]
One then obtains
\[
{\cal L}(\bar I)=C(0)-C(\sigma_h)=\frac{1}{\beta S_h \bar I}\int_{\bar \sigma(\bar I)}^{\sigma_h} f(\sigma)\,d\sigma
 \]
 and with \eqref{barsigma} one can write
 \[
 {\cal L}(\bar I)=\frac{1}{\beta S_h \bar I}\left(\int_0^{\sigma_h} f(\sigma)\,d\sigma +I_0-\bar I\right)
 \]
 On another hand, one has
 \[
 \int_0^{\sigma_h} f(\sigma)\,d\sigma=\sigma_h+S_h\log\left(\frac{S_h}{S_0}\right)=I_h-I_0
 \]
 which finally gives the expression \eqref{calL}.
	\end{proof}

\bigskip

Then, the best admissible NSN control can be given as follows. 

\begin{coro}
\label{coro}
	When $Q\leq\frac{I_h-I_0}{\beta S_h I_0}$, the smallest $\bar I \in [I_0,I_h]$ for which the solution with the NSN strategy is admissible is given by the value
	\begin{equation}
	    \label{Ibarstar}
	\bar I^\star(Q):=\frac{I_h}{Q\beta S_h+1}
	\end{equation}
	and one has
	\begin{equation}
	    \label{LIbar}
	{\cal L}(\bar I^\star(Q))=Q
	\end{equation}
\end{coro}

\bigskip

We give now our main result that shows that the NSN strategy is optimal.

\begin{prop}
	\label{propIbar}
Let Assumptions \ref{assumR0} and \ref{assumS0} be fulfilled. Then, the NSN feedback is optimal with
\[
\bar I = 
\begin{cases}
	\bar I^\star(Q), & Q< \frac{I_h-I_0}{\beta S_h I_0}\\
	I_0, & Q \geq \frac{I_h-I_0}{\beta S_h I_0}
\end{cases}
\]
where $\bar I^\star(Q)$ is defined in \eqref{Ibarstar}, and $\bar I$ is the optimal value of problem \eqref{pbpeak}.
\end{prop}

\begin{proof}
	When $Q \geq \frac{I_h-I_0}{\beta S_h I_0}$, the NSN strategy is admissible and the corresponding solution verifies
	\[
	\max_{t \geq 0} I(t) = I_0
	\]
	which is thus optimal. 
	
	Consider now $Q < \frac{I_h-I_0}{\beta S_h I_0}$. Let $(S^\star(\cdot), I^\star(\cdot),C^\star(\cdot))$ be the solution generated by the NSN strategy with $\bar I=\bar I^\star(Q)$, and denote $u^\star(\cdot)$ the corresponding control.
	Let
	\[
   \bar S:=S^\star(\bar t) \mbox{ where } 	\bar t = \inf \{ t >0, \; I^\star(t)=\bar I \}
	\]
	and
	\[
	t_h^\star := \inf \{ t  > \bar t, \; S^\star(t)=S_h \},
	\]
   We consider in the $(S,I)$ plane the curve
   \[
   {\cal C}^\star:=\{  (S^\star(t),I^\star(t)) ; \; t \in [0,t_h^\star]\}
   \]
   For $S \geq \bar S$, the control \eqref{feedback} is null and a upward normal to ${\cal C}^\star$ is given by the expression
   \[
   \vec n(S,I)=\left[\begin{array}{c}
   \beta SI-\gamma I\\
   \beta SI
   \end{array}\right], \;
   (S,I)\in {\cal C}^\star \mbox{ with } S \in [\bar S,S_0]
   \]
   On another hand, the vector field in the $(S,I)$ plane of any admissible solution is
   \[
   \vec v(S,I,u)=\left[\begin{array}{c}
   -\beta SI(1-u)\\
   \beta SI(1-u)-\gamma I
   \end{array}\right]
   \]
   Then, one has
   \[
   \vec n(S,I).\vec v(S,I,u)=-\beta\gamma SI^2 u \leq 0 ,  \;
   (S,I)\in {\cal C}^\star \mbox{ with } S \in [\bar S,S_0]
   \]
   which shows that any admissible solution is below the curve ${\cal C}^\star$ in the $(S,I)$ plane for $S \in [\bar S,S_0]$. 
   For $S \in [S_h,\bar S]$, the curve ${\cal C}^\star$ is an horizontal line with $I=\bar I$. Therefore, if there exists an admissible solution $(S(\cdot),I(\cdot),C(\cdot))$ with $\max_t I(t)< \bar I$, its trajectory in the $(S,I)$ plane has to be below the curve ${\cal C}^\star$ for any $S \in [S_h,S_0]$. Let
   \[
   t_h :=  \inf \{ t > 0, \; S(t)=S_h \}
   \]
   One has thus $I(t_h) < \bar I$. Define
   \[
   T:=t_h^\star+\frac{1}{\gamma}\log\left(\frac{\bar I}{I(t_h)}\right) > t_h^\star
   \]
   and consider the (non-admissible) solution $(\tilde S(\cdot),\tilde I(\cdot),\tilde C(\cdot))$ of \eqref{sys} on $[0,T]$ defined by the control
   \[
   \tilde u(t)=\begin{cases}
   u^\star(t) , & t \in [0,t_h^\star)\\
   1 , & t \in [t_h^\star,T]
   \end{cases}
   \]
   One can straightforwardly check with equations \eqref{sys} that the solution is
   \[
   (\tilde S(t),\tilde I(t),\tilde C(t))=\begin{cases}
   (S^\star(t),I^\star(t),C^\star(t)), & t \in [0,t_h^\star)\\
   (S_h,\bar I\exp(-\gamma(t-t_h^\star)),C^\star(t_h^\star)+t_h^\star-t), &  t \in [t_h^\star,T]  
   \end{cases}
   \]
   Remind, from Corollary \ref {coro}, that one has $C^\star(t_h^\star)=0$ by equation \eqref{LIbar}).
   Clearly, one has $(\tilde S(T),\tilde I(T))=(S_h,I(t_h))$ and $\tilde C(T)<0$. 
   We consider now in the $(S,I)$ plane the simple closed curve $\Gamma$ which is the concatenation of the trajectory $(\tilde S(\cdot),\tilde I(\cdot))$ on forward time with the trajectory $(S(\cdot),I(\cdot))$ in backward time:
   \[
   \Gamma:= \{  (\tilde S(\tau),\tilde I(\tau)), \; \tau \in [0,T] \} \cup \{  (S(T+t_h-t),I(T+t_h-t)), \; \tau \in [T,T+t_h] \}
   \]
   that is anticlockwise oriented by $\tau \in [0,T+t_h]$.  Then one has
   \[
   \tilde C(T)-C(t_h)=\oint_{\Gamma} dC
   \]
   From equations \eqref{sys}, one gets
   \[
   dC=-\frac{dS}{\beta SI}-dt=-\frac{dS}{\beta SI}+\frac{dS+dI}{\gamma}=\left(1-\frac{S_h}{S}\right)\frac{dS}{\gamma I}+ \frac{dI}{\gamma I}
   \]
   and thus
   \[
   \tilde C(T)-C(t_h)=\oint_{\Gamma} P(S,I)dS+Q(S,I)dI
   \]
   with
   \[
   P(S,I)=\left(1-\frac{S_h}{S}\right)\frac{1}{\gamma I}, \quad Q(S,I)=\frac{1}{\gamma I}
   \]
   By the Green's Theorem, one obtains
   \[
   \tilde C(T)-C(t_h)=\iint_{\cal D} \left(\frac{\partial Q}{\partial S}(S,I)-  \frac{\partial P}{\partial I}(S,I)\right)\,dSdI=\iint_{\cal D}\left(1-\frac{S_h}{S}\right)\frac{1}{\gamma I^2}\,dSdI > 0
   \]
   where ${\cal D}$ is the domain bounded by $\Gamma$ (see Figure \ref{figGreen} as an illustration). This implies $C(t_h)<\tilde C(T)<0$ and thus a contradiction with the admissibility condition \eqref{target} of the solution $(S(\cdot),I(\cdot),C(\cdot))$. We conclude that $(S^\star(\cdot),I^\star(\cdot),C^\star(\cdot))$ is optimal.
\end{proof}

\begin{figure}[!ht]
	\centering
	\includegraphics[scale=0.6]{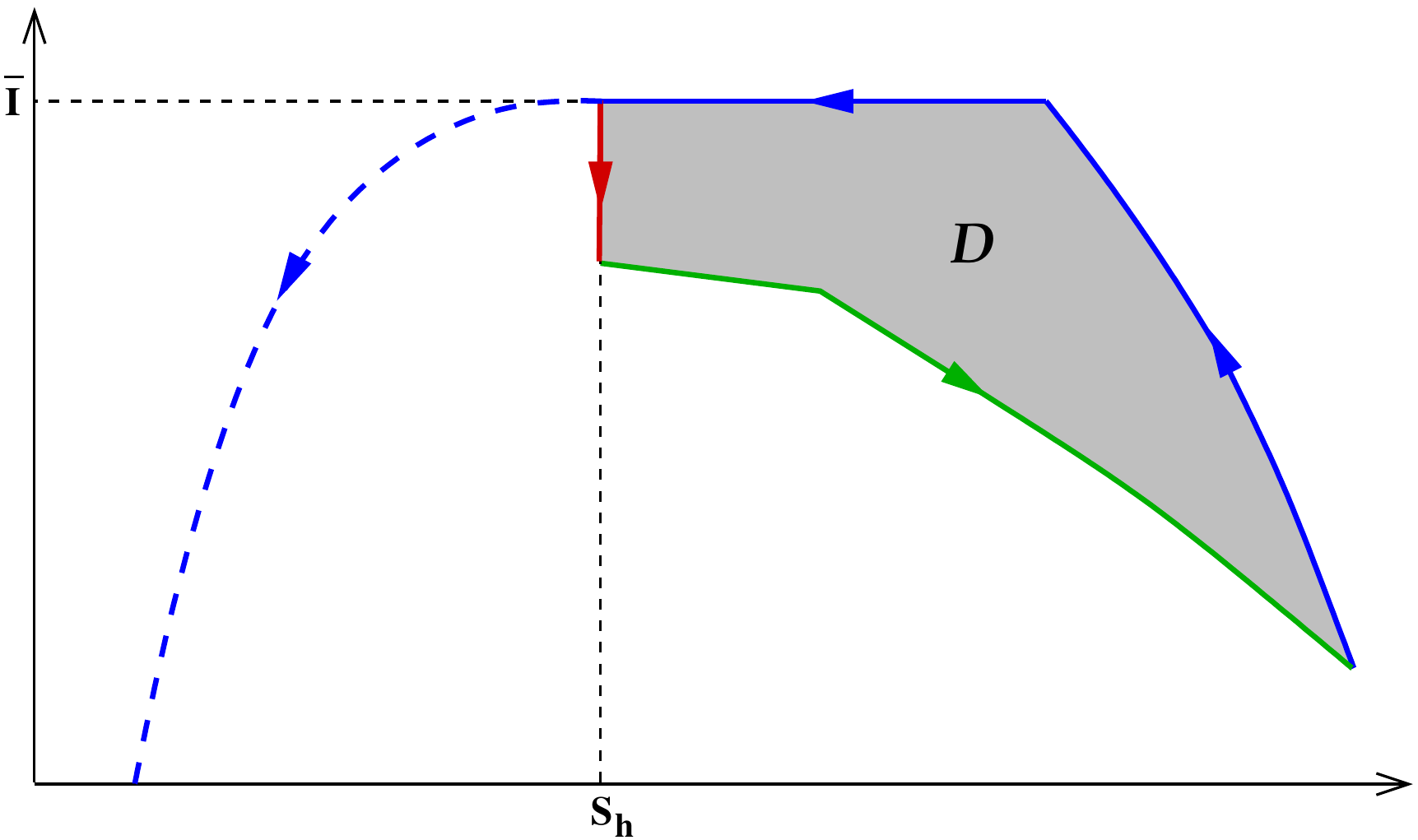}
	\caption{The closed curve $\Gamma$ is composed of the trajectory $(S^\star(\cdot),I^\star(\cdot))$ in blue up to to the point $(S_h,\bar I)$, the additional part $(\tilde S(\cdot),\tilde I(\cdot))$ in red and an hypothetical better trajectory $(S(\cdot),I(\cdot))$ in backward time in green.}
		\label{figGreen}
\end{figure}

\section{Numerical illustrations and discussion}
\label{sectionSimu}

We have considered the same parameters and initial condition as in \cite{Morris} (see Table \ref{tablepar}).
\begin{table}[ht!]
	\begin{center}
		\begin{tabular}{c|c|c|c}
			$\beta$ &  $\gamma$ & $S(0)$ & $I(0)$\\
			\hline\hline
			$\ds \vphantom{\Big(}$
			$0.21$  & $0.07$ &  $1-10^{-6}$ & $10^{-6}$
		\end{tabular}
		\caption{SIR parameters and initial condition}
		\label{tablepar}
	\end{center}
\end{table} 
For these values, one computes
\[
{\cal R}_0=3, \quad S_h=\frac{1}{3}, \quad I_h \simeq 0.3
\]
Figure \ref{figsimu} presents a simulation of the optimal solution for the budget $Q=28$, as an example (the minimum peak is reached for $\bar I \simeq 0.1015$).
\begin{figure}[!ht]
	\centering
	\includegraphics[scale=0.6]{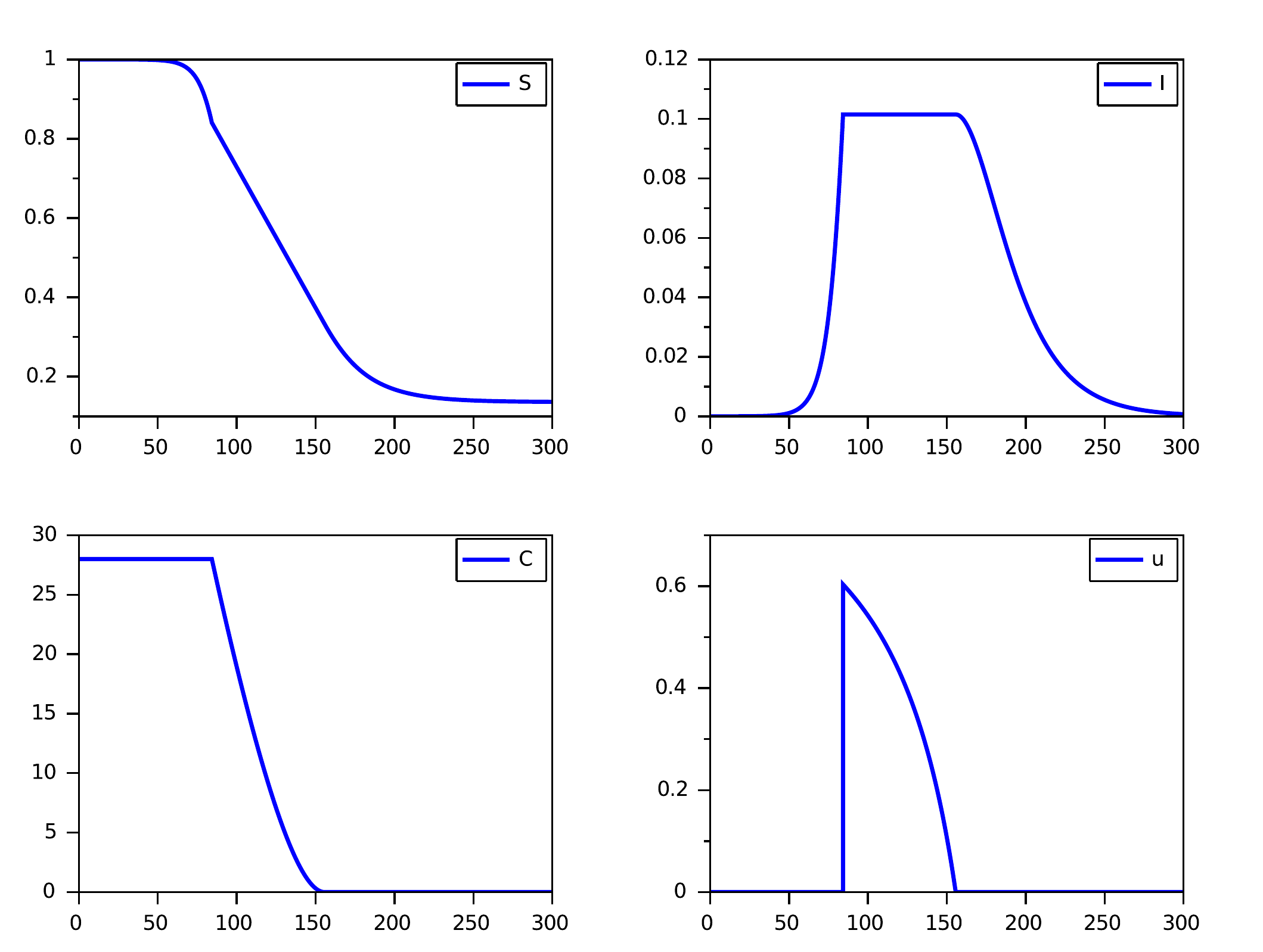}
	\caption{Optimal solution for $Q=28$.}
	\label{figsimu}
\end{figure}
As a comparison, the optimal strategy obtained by Morris et al.~in \cite{Morris} for a fixed time duration of interventions without consideration of any budget is quite different (see Figure \ref{figsame}). It consists in four phases: no intervention, maintain $I$ constant, apply the maximal control (i.e.~$u=1$) and stop the intervention. This control presents thus three switches and relies on a full break of the transmission, differently to the NSN strategy which presents only one switch (see Remark \ref{remswitch}) and does not require a full break (see the maximal value of the control given in Lemma \ref{lemIbar}). 
\begin{figure}[!ht]
	\centering
	\includegraphics[scale=0.6]{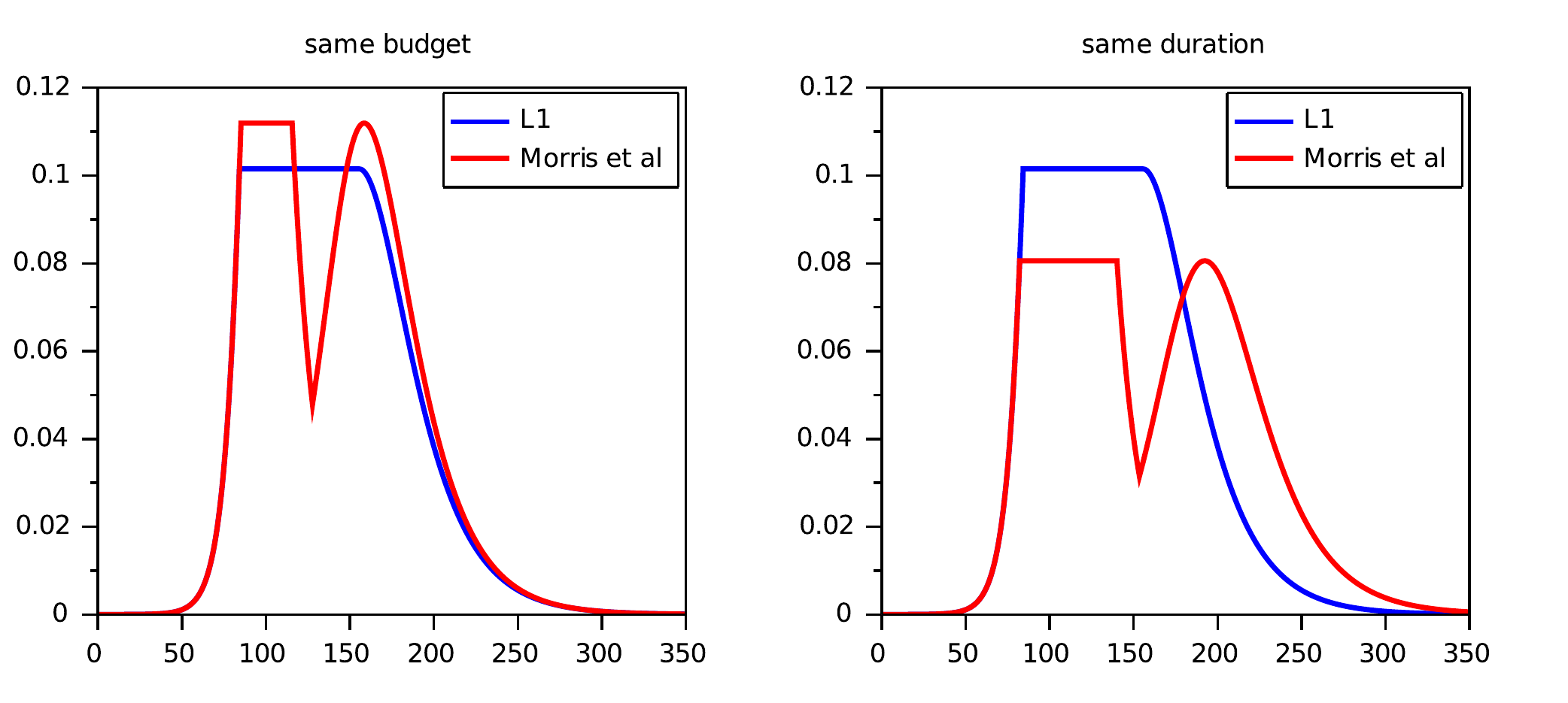}
	\caption{Comparison of the time evolution of the infected population $I$ between the optimal NSN strategy and the optimal one of Morris et al.}
	\label{figsame}
\end{figure}
Applying an NSN strategy appears thus less restrictive to be applied in practice. The strategy proposed by Morris et al.~induces also a second peak: after the third phase, the prevalence $I$ increases again up to a peak which has to be equal to the level maintained during the second phase if its optimally chosen. But this second peak turns out to be non robust under a mischoice (or mistiming) of the second phase (see \cite{Morris} for more details). Comparatively, the NSN is naturally robust with respect to a bad choice of $\bar I$: the maximum value of $I$ is always guaranteed to be equal to $\bar I$. However, a mischoice of $\bar I$ has an impact on the budget of the NSN strategy, given by expression \eqref{calL} and illustrated in Table \ref{tableLIbar} (for model parameters given in Table \ref{tablepar} and $Q=28$).
\begin{table}[ht!]
	\begin{center}
		\begin{tabular}{c|ccccccc}
			$\bar I -\bar I^\star$ &  $-10\%$ & $-5\%$ & $-1\%$ & $+1\%$ & $+5\%$ & $+10\%$\\
			\hline\hline 
			 $\ds \vphantom{\Big(}$
			 ${\cal L}(\bar I)-Q$ &
			$+17\%$  & $+8\%$ &  $+1.5\%$ & $+1.5\%$ & $-7\%$  & $-14\%$
			\end{tabular}
		\caption{Variation of the control budget of the NSN strategy under a mischoice of $\bar I$}
		\label{tableLIbar}
	\end{center}
\end{table} 

In case of a new epidemic among a large population, one can consider that the initial number of infected individuals is very low, while all the remaining population is susceptible. Therefore, one has $S_0+I_0=1$ with $I_0$ very small, and the optimal value of $\bar I$ can be well approximated by its limiting expression for $I_0=0$, that is
\begin{equation}
    \label{Ibarlim}
\bar I_{\ell}:=\frac{1-S_h+S_h\log(S_h)}{Q\beta S_h+1}
\end{equation}
From property \eqref{invariant}, one also gets an approximation of the value $\bar S_{\ell}$ of $S$ when $I$ reaches $\bar I_{\ell}$ with $u=0$, as the solution of the equation
\[
\bar S_{\ell}+\bar I_{\ell}-S_h \log(\bar S_{\ell})=1
\]
and then an approximation of the duration of the intervention is given by
\[
d_{\ell}:=\frac{S_h-\bar S_{\ell}}{\gamma \bar I_{\ell}}
\]
(one can easily check that along the singular arc $I=\bar I$, one has $\dot S=-\gamma \bar I$).
For the parameters of Table \ref{tablepar}, one obtains the limiting values given in Table \ref{tablelim}.
This means that depending on the budget $Q$ only, one can determine the minimal peak and the optimal strategy to apply, without the knowledge of the initial size of the infected population, provided that parameters $\beta$ and $\gamma$ of the disease are known. 
\begin{table}[ht!]
	\begin{center}
		\begin{tabular}{c|c|c}
			$\bar I_{\ell}$ &  $\bar S_{\ell}$ & $d_{\ell}$\\
			\hline\hline
			$\ds \vphantom{\Big(}$
			$0.1015$  & $0.8406$ &  $71.39$
		\end{tabular}
		\caption{The limiting optimal values for arbitrarily small $I_0$ (with $Q=28$)}
		\label{tablelim}
	\end{center}
\end{table} 

\noindent The question of parameters estimation in the SIR model from data is out of the scope of the present work. However, while reaching $I=\bar I_{\ell}$ without intervention, one may expect refinement of the estimates and thus an adjustment of the value of $\bar I_{\ell}$. 

Note that if it is rather the height of the peak $\bar I$ that is imposed, the corresponding effort can be determined with expression \eqref{Ibarlim}, that is
\[
Q=\frac{1}{\beta S_h}\left(\frac{1-S_h}{\bar I}-1\right)
\]
as well with the duration of the intervention.

\medskip

To have a better insight of the impacts of the available budget $Q$ on the course of the epidemic, we have considered four characteristics numbers:
\begin{itemize}
	\item $t_i$: the starting date of the intervention,
	\item $d$: the duration of the intervention,
	\item $\bar I$: the height of the peak,
	\item $u_{max}$: the maximal value of the control,
\end{itemize}
of the optimal solution, depicted on Figure \ref{figQ} as a function of $Q$ for $I_0=10^{-6}$ and $S_0+I_0=1$.
Let us note that the maximal budget $Q$ under which it is not possible to immediately slow down the progress of the epidemic is given, according to Proposition \ref{propIbar}, by
\[
Q_{max}:=\frac{I_h-I_0}{\beta S_h I_0} \simeq 4.3 \, 10^6
\]
which is quite high.
\begin{figure}[!ht]
	\centering
	\includegraphics[scale=0.6]{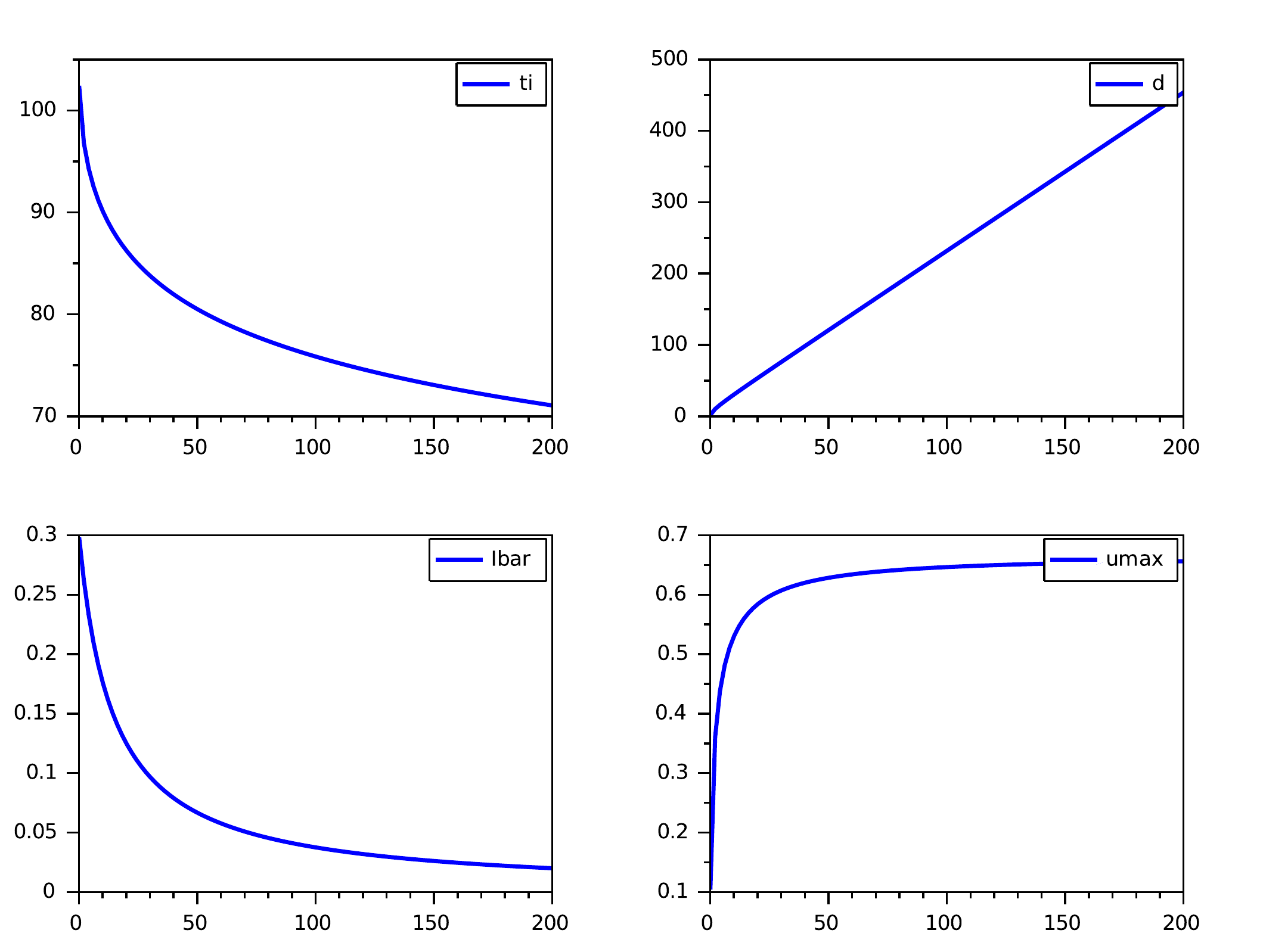}
	\caption{Characteristics numbers as functions of $Q$.}
	\label{figQ}
\end{figure}
Moreover, the maximal value of the control is bounded by the value
\[
u_{max}(\bar I) \leq 1 - S_h = \frac{2}{3}
\]
far from the value $1$ (that would consists in a total lockdown of the population). On Figure \ref{figQ}, one can see that the peak $\bar I$ can be drastically reduced under a reasonable budget, and that taking larger budgets slows down the decrease of the peak, while the duration of the intervention carries on increasing, almost linearly. 
Indeed, remind that one has $d=(\bar S-S_h)/(\gamma\bar I)$ and for an optimal value of $\bar I$, one has $Q=(I_h-\bar I)/(\gamma\bar I)$ from \eqref{Ibarstar}. Then one gets
\[
d=\frac{\bar S-S_h}{I_h - \bar I}\, Q 
\]
but for large values of $Q$, $\bar I$ is small and $\bar S$ closed to one, which gives an approximation of $d$ as the linear function of $Q$
\[
d \simeq \frac{1-S_h}{I_h}\, Q \simeq 2.194 \, Q
\]
This implies that for a long duration, fixing the budget $Q$ or the duration $d$ tends to be equivalent. Therefore, for a same large duration, the optimal peak gets closed from the optimal one of the strategy of Morris et al.~which constraints the duration only, but the difference of the budgets of these two strategies gets increasing with always a lower one for the NSN strategy, as one can see on Figure \ref{figcompar}. 
\begin{figure}[!ht]
	\centering
	\includegraphics[scale=0.6]{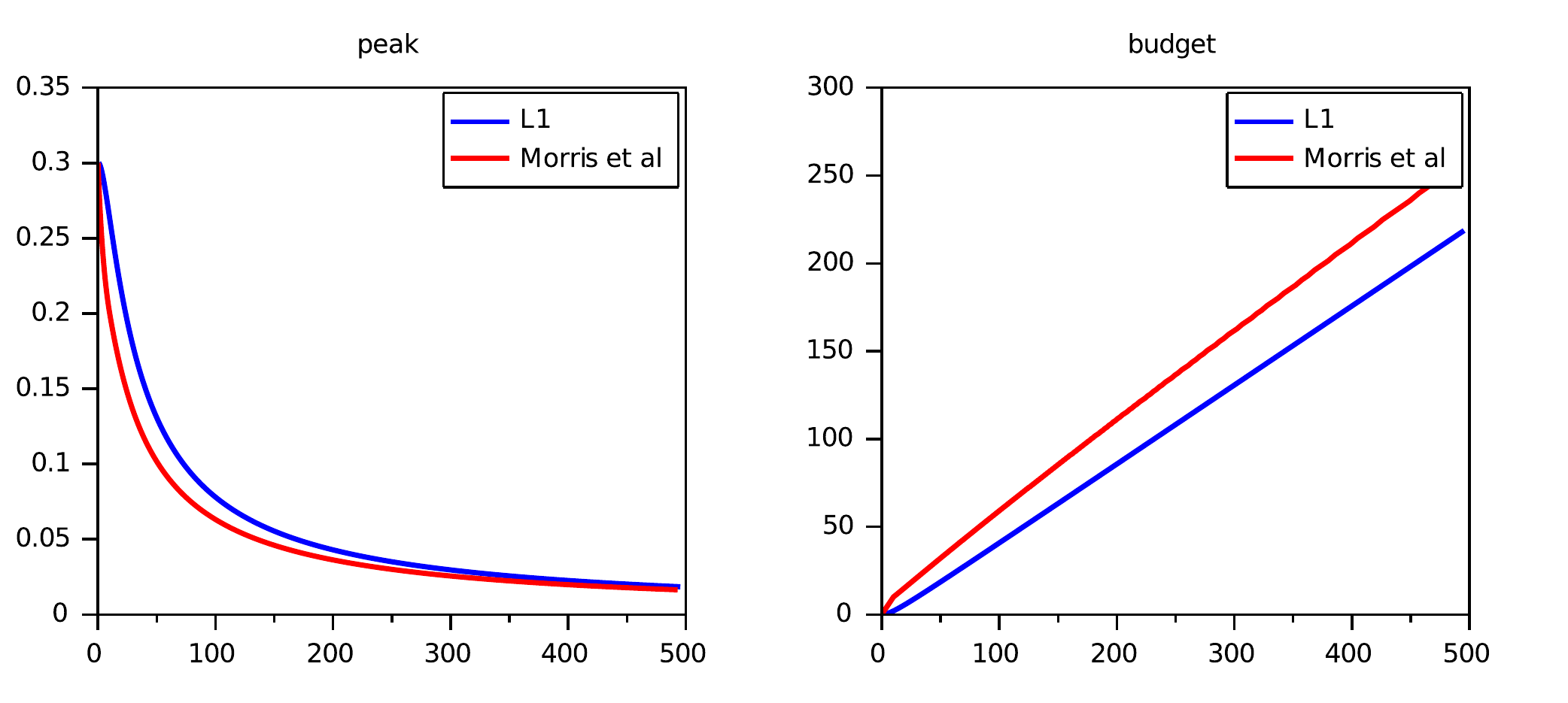}
	\caption{Comparison of the performances of the optimal strategies with same duration (in abscissa).}
	\label{figcompar}
\end{figure}

\medskip

Finally, this analysis highlights (as already mentioned in \cite{Morris,Lobry}) the importance to do not intervene too early (unless one has a very large budget) and to choose the "right" time to launch interventions. We believe that curves as in Figure \ref{figQ} might be of some help for decision makers.

\section*{Acknowledgments}
The authors are very grateful to Professors Claude Lobry and Hector Ramirez for fruitful exchanges.
The first author thanks the support of ANID-PFCHA/Doctorado Nacional/2018-21180348,
FONDECYT grant 1201982 and Centro de Modelamiento Matem\'atico (CMM), ACE210010 and FB210005, BASAL funds for center of excellence, all of them from ANID (Chile).


	\end{document}